\documentclass{amsart}

\usepackage{amsmath}
\usepackage{amsthm}
\usepackage{amssymb}
\usepackage{amsbsy}
\usepackage{amsfonts}
\usepackage{amstext}
\usepackage{amscd}
\usepackage{graphicx}

\usepackage{color}
\numberwithin{equation}{section}
\theoremstyle{plain}
\newtheorem{thm}{Theorem}[section]
\newtheorem{prop}[thm]{Proposition}
\newtheorem{cor}[thm]{Corollary}
\newtheorem{lem}[thm]{Lemma}
\theoremstyle{definition}
\newtheorem{exa}[thm]{Example}

\newtheorem{rem}[thm]{Remark}
\newtheorem{defi}[thm]{Definition}

\newcommand{\nn}{\mathbb{N}}

\begin{document}
\title[Distance-$k$ graphs of free product]
{On the spectral distributions of distance-$k$ graph of free product graphs.}

\author{Octavio Arizmendi}
\author{Tulio Gaxiola}
\address{Department of Probability and Statistics, CIMAT, Guanajuato, Mexico}
\email{octavius@cimat.mx, marco.gaxiola@cimat.mx}

\date{\today}
\maketitle

\begin{abstract}
We calculate the distribution with respect to the vacuum state of the distance-$k$ graph of a $d$-regular tree. 
From this result we show that the distance-$k$ graph of a $d$-regular graphs converges to the distribution of the distance-$k$ graph of a regular tree.  
 Finally, we prove that, properly normalized, the asymptotic distributions of distance-$k$ graphs of the $d$-fold free product graph, as $d$ tends to infinity, is given by the distribution of $P_k(s)$, where $s$ is a semicircle random variable and $P_k$ is the $k$-th Chebychev polynomial.  
\end{abstract}

\section{Introduction}

In this paper we consider three problems on the distance-$k$ graphs, which generalize results of Kesten \cite{Kes} (on random walks on free groups), McKay \cite{Mc} (on the asymptotic distribution of $d$-regular graphs) and the free central limit of Voiculescu \cite{Voi}. The first one is finding, for fixed $d$, the distribution w.r.t. the vacuum state of the distance-$k$ graphs of a $d$-regular tree.
Then we consider two related problems which are in the asymptotic regime. On one hand, we show that the asymptotic distributions of distance-$k$ graphs of $d$-fold free product graphs, as $d$ tends to infinity, are
given by the distribution of $P_k(s)$, where $s$ is a semicircle distribution and $P_k$ is the $k$-th Chebychev polynomial. On the other hand, we find the asymptotic spectral distribution of the distance-$k$ graph of a random $d$-regular graph of size $n$, as $n$ tends to infinity.

More precisely our first result is the following.
\begin{thm}\label{Tuno}
For $d\geq 2,~ k\geq1$, let $A^{[k]}_d$ be the adjacency matrix of distance-$k$ graph of the $d$-regular tree. Then the distribution with
 respect to the vacuum state of $A^{[k]}_d$ is given by the probability distribution of
 \begin{equation*}
  T_k(b)=\sqrt{\frac{d-1}{d}}P_k\left(\frac{b}{2\sqrt{d-1}}\right)-\frac{1}{\sqrt{d(d-1)}}P_{k-2}\left(\frac{b}{2\sqrt{d-1}}\right)~~~~,
 \end{equation*}
where  $P_k$ is the Chebyshev polynomial of order $k$ and $b$ is a random variable with Kesten-McKay distribution,  $\mu_d$.
\end{thm}

The spectrum of the distance-$k$ graph of the Cartesian product of graphs was first studied by Kurihara and Hibino \cite{KH} where they consider the distance-$2$ graph of $K_2 \times \cdots \times K_2$ (the $n$-dimensional hypercube).
More recently, in a series of papers \cite{Hi,HLO,K,KH,LO,O2}  the asymptotic spectral distribution of the distance-$k$ graph of the $N$-fold power
of the Cartesian product was studied.  These investigations, finally lead to the following theorem which generalizes the central limit theorem for Cartesian products of graphs.

\begin{thm}[Hibino, Lee  and Obata \cite{HLO}] \label{HLO}
Let $G=(V,E)$ be a finite connected graph with $|V|\ge2$.
For $N\ge1$ and $k\ge1$ let $G^{[N,k]}$ be the distance-$k$ graph of $G^N=G\times \cdots\times G$ 
($N$-fold Cartesian power)
and $A^{[N,k]}$ its adjacency matrix.
Then, for a fixed $k\ge1$, 
the eigenvalue distribution of $N^{-k/2}A^{[N,k]}$ converges in moments as $N\rightarrow\infty$
to the probability distribution of
\begin{equation}
\left(\frac{2|E|}{|V|}\right)^{k/2}
\frac{1}{k!}\Tilde{H}_k(g),
\end{equation}
where $\Tilde{H}_k$ is the monic Hermite polynomial of degree $k$ and $g$ is a random variable obeying the standard normal distribution $\mathcal{N}(0,1)$.
\end{thm}

 In the same spirit, in \cite{AG},  we consider the analog of Theorem \ref{HLO} by changing the Cartesian product by the star product.

\begin{thm}[Arizmendi and Gaxiola \cite{AG}] \label{T1} Let $G=(V,E,e)$ be a locally finite connected graph and let $k\in\mathbb{N}$ be such that  $G^{[k]}$ is not trivial.  For $N\ge1$ and $k\ge1$ let $G^{[\star N,k]}$ be the distance-$k$ graph of $G^{\star N}=G\star \cdots\star G$  ($N$-fold star power)
and $A^{[\star N,k]}$ its adjacency matrix. Furthermore, let $\sigma=V_{e}^{\left[ k\right] }$ be the number of neighbors of $e$ in the distance-$k$ graph of $G,$
then the distribution with respect to the vacuum state of $(N\sigma)^{-1/2}A^{[\star N,k]}$ converges in distribution as $N\rightarrow\infty$ to a centered Bernoulli distribution. That is,
\begin{equation*}\label{limit}
\frac{A^{\left[ \star N,k\right] }}{\sqrt{N\sigma}}
\longrightarrow \frac{1}{2} \delta _{-1}+\frac{1}{2}\delta _{1},
\end{equation*}
weakly.
\end{thm}

Our second theorem is the free counterpart of the theorems above.

\begin{thm} \label{T2} Let $G=(V,E,e)$ be a finite connected graph and let $k\in\mathbb{N}$.  For $N\ge1$ and $k\ge1$ let $G^{[*N,k]}$ be the distance-$k$ graph of $G^{*N}=G*\cdots *G$  ($N$-fold free power)
and $A^{[*N,k]}$ its adjacency matrix. Furthermore, let $\sigma$ be the number of neighbors of $e$ in the graph $G$. Then the 
distribution with respect to the vacuum state of $(N\sigma)^{-k/2}A^{[*N,k]}$ converges in moments (and then weakly) as $N\to\infty$ to the probability
distribution of 
\begin{equation}
  P_k(s),
 \end{equation}
 where $P_k$ is the Chebychev polynomial of order $k$ and $s$ is a random variable obeying the semicircle law.
\end{thm}

Finally, our third theorem considers the asymptotic spectral distribution of the distance-$k$ graph of $d$-regular random graphs. 

\begin{thm}\label{T3}
 Let $d$, $k$ be fixed integers and, for each $n$, let $F_n(x)$  be the expected eigenvalue distribution of the distance-$k$ graph of a random regular graph with degree $d$ and order $2n$.  Then, as $n$ tends to infinity, $F_n(x)$ converges to the distribution of $A^{[k]}_d$  with respect to the vacuum state, described in Theorem 1.1.
\end{thm}

  Apart from this introduction the paper is organized as follows. In Section 2 we give the basic preliminaries on graphs, orthogonal polynomials and Non-Commutative Probability and Kesten-McKay distributions.  Section 3  is devoted to prove Theorem \ref{T1}.  We prove Theorem \ref{T2} in Sections 4 and 5. Section 4 considers the case $k=2$, while Section 5 considers the case $k\geq3$. Finally, in Section 6 we use the results of Section 3 to prove Theorem \ref{T3}. 

\section{Preliminaries}

In this section we give very basic preliminaries on graphs, free product graphs, orthogonal polynomials, Jacobi parameters and non-commutative
probability. The reader familiar with these objects may skip this section.

\subsection{Graphs}
By a \textit{rooted graph} we understand a pair $(\mathcal{G},e)$, where $\mathcal{G}=(V,E)$, is a undirected graph with set of vertices
$V=V(\mathcal{G})$, and the set of edges $E=E(\mathcal{G})\subseteq \{(x,x'):x,\ x'\in V,\ x\neq x'\}$ and $e\in V$ is
a distinguished vertex called the \textit{root}. For rooted graphs we will use the notation $V^0=V\backslash \{e\}$.
Two vertices $x,x'\in V$ are called \textit{adjacent} if $(x,x')\in E$,  i.e.
vertices $x,x'$ are connected with an edge. Then we write $x \sim x'$.
Simple graphs have no loops, i.e. $(x,x)\notin E$ for all $x\in V$. A graph is called \textit{finite} if $|V|<\infty$.
The \textit{degree} of $x\in V$ is defined by $\kappa(x)=|\{x'\in V:x'\sim x\}|$,
where $|I|$ stands for the cardinality of $I$.
A graph is called \textit{locally finite} if $\kappa(x)<\infty$ for every $x\in V$.
It is called \textit{uniformly locally finite} if ${\rm sup}\{\kappa(x): x\in V\}<\infty$.

We define the \textit{free product} of the rooted vertex sets $(V_i,e_i)$,
 $i\in I$, where $I$ is a countable set, by the rooted set $(*_{i\in I}V_i, e)$, where
$$*_{i\in I}V_i=\{e\}\cup \{v_1 v_2 \cdots v_m : v_k\in V_{i_k}^0,\ \text{and}\ i_1\neq i_2\neq \cdots \neq i_m,\ m\in \mathbb{N}\},$$
and $e$ is the empty word.

\begin{defi}
 The \textit{free product of rooted graph} $(\mathcal{G}_i,e_i),\ i\in I$, is defined by the rooted graph $(*_{i\in I}\mathcal{G}_i,e)$
 with vertex set $*_{i\in I}V_i$ and edge set $*_{i\in I}E_i$, defined by
 $$*_{i\in I}E_i:= \{(vu,v'u):(v,v')\in \bigcup_{i\in I}E_i\ \text{and}\ u,\ vu,\ v'u\in *_{i\in I}V_i\}.$$
 We denote this product by $*_{i\in I}(\mathcal{G}_i,e_i)$ or $*_{i\in I}\mathcal{G}$ if no confusion arises.
If $I=[n]$, we denote by $G^{*n}=(*_{i\in I}G,e)$.
\end{defi}

Notice that for a fixed word $u=v_1 v_2 \cdots v_m$ with $j\in I$ with $v_1\notin V_j$ the subgraph of $(*_{i\in I}\mathcal{G_i}_i,e)$  induced by the vertex set $\{wu: w\in V_{j}\}$ is isomorphic to $G_j$. This motivates the following definition
\begin{defi}If $x,y\in*_{i\in I}V_i$, we say that $x$ and $y$ are in the same copy of $G_i$ if $x=vu$ and $y=v'u$ for some $u\in*_{i\in I}V_i$ and $v,v'\in V_j^{0}$ for some $j\in I$.
\end{defi}

For a given graph $G=(V,E)$, its \textit{distance-$k$ graph} $G^{[k]}=(V,E^{[k]})$ is 
defined by $$E^{[k]}=\{(x,y):x,y\in V,\ \partial_G(x,y)=k\}.$$

For $x\in V$, let $\delta(x)$ be the indicator function of the one-element set $\{x\}$.
Then $\{\delta(x),\,~x\in V\}$ is an orthonormal basis of the Hilbert space $l^{2}(V)$ of
square integrable functions on the set $V$, with the usual inner product.

The \textit{adjacency matrix} $A=A({\mathcal G})$ of ${\mathcal G}$ is a 0-1 matrix defined by
\begin{equation}
A_{x,x'}=\left\{
\begin{array}{ll}
1 & {\rm if} \;\; x\sim x'\\
0 & {\rm otherwise.}
\end{array}
\right.
\end{equation}
We identify $A$ with the densely defined symmetric operator on $l^{2}(V)$ defined by
\begin{equation}
A\delta(x)=\sum_{x\sim x'}\delta(x')
\end{equation}
for $x\in V$. Notice that the sum on the right-hand-side is finite since our graph is assumed
to be locally finite. It is known that $A(\mathcal{G})$ is bounded if and only if ${\mathcal G}$ is uniformly
locally finite. If $A(\mathcal{G})$ is essentially self-adjoint, its closure is called the
\textit{adjacency operator} of ${\mathcal G}$ and its spectrum is called the spectrum of ${\mathcal G}$.

The unital algebra generated by $A$, i.e. the algebra of polynomials in $A$, is called
the \textit{adjacency algebra} of ${\mathcal G}$ and is denoted by ${\mathcal A}({\mathcal G})$ or simply
${\mathcal A}$.

\subsection{Orthogonal Polynomials and The Jacobi Parameters}

Let $\mu$ be a probability measure with all moments, that is $m_n(\mu):=\int_{\mathbb{R}}|x^{n}|\mu (dx)<\infty$. The Jacobi parameters  $\gamma _{m}=\gamma _{m}(\mu )\geq 0,\beta
_{m}=\beta _{m}(\mu )\in\mathbb{R}$, are defined by the recursion 
\begin{equation*}
xQ_{m}(x)=Q_{m+1}(x)+\beta _{m}Q_{m}(x)+\gamma _{m-1}Q_{m-1}(x),
\end{equation*}
where the polynomials  $Q_{-1}(x)=0,$ $Q_{0}(x)=1$ and $(Q_{m})_{m\geq 0}$ is a sequence of orthogonal monic polynomials with respect to $\mu $, that is,
\begin{equation*}
\int_{\mathbb{R}}Q_{m}(x)Q_{n}(x)\mu (dx)=0\text{ \  \  \  \ if }m\neq n.
\end{equation*}

\begin{exa} The Chebyshev polynomials of the second kind are defined by the recurrence relation
$$P_0(x)=1,\ \ \ P_1(x)=x,$$
and
\begin{equation}\label{chev} xP_n(x)=P_{n+1}(x)+P_{n-1}(x)\ \ \ \forall n\geq 1.\end{equation}
These polynomials are orthogonal with respect to the semicircular law, which is defined by the density
$$\mathbf{d}\mu=\frac{1}{2\pi}\sqrt{4-x^2}\mathbf{d}x.$$ The Jacobi parameters of $\mu$ are $\beta_m=0$ and $\gamma_m=1$ for all $m\geq0$.
\end{exa}

\subsection{Non-Commutative Probability Spaces}

A $C^*$\textit{-probability space} is a pair $(\mathcal{A},\varphi)$, where $\mathcal{A}$ is a unital $C^*$-algebra and $\varphi:\mathcal{A}\to\mathbb{C}$ is a positive unital linear functional. The elements of $\mathcal{A}$ are called (non-commutative) random variables. An element $a\in\mathcal{A}$ such that $a=a^*$ is called self-adjoint.

The functional $\varphi$ should be understood as the expectation in classical probability.

 For $a_1,\dots,a_k\in \mathcal{A}$, we will refer to the values of $\varphi(a_{i_1}\cdots a_{i_n})$, $1\leq i_1,...i_{n}\leq k$, $n\geq1$, as the \textit{joint moments} of $a_1,\dots,a_k$.  If there exists $1\leq m,l,\leq n$ with $i(m)\neq i(l)$ we call it a \textit{mixed moment}.

For any self-adjoint element $a\in\mathcal{A}$ there exists a unique probability measure $\mu_a$ (its spectral distribution) with the same moments as $a$, that is, $$\int_{\mathbb{R}}x^{k}\mu_a (dx)=\varphi (a^{k}), \quad \forall k\in \mathbb{N}.$$

We say that a sequence $a_n\in\mathcal{A}_n$ \emph{converges in distribution} to $a\in\mathcal{A}$ if $\mu_{a_n}$ converges in distribution to $\mu_a$. 
In this setting convergence in distribution is replaced by convergence in moments. Let $(\phi_n,\mathcal{A}_n)$ be a sequence of $C^*$-probability spaces and let $a\in(\mathcal{A},\varphi)$ be a selfadjoint random variable. We say that the sequence $a_n\in(\phi_n,\mathcal{A}_n)$ of selfadjoint random variables converges to $a$ \emph{in moments} if
$$\lim_{n\to \infty} \phi_n(a_n^k)=\phi(a^k) ~~\text{for all } k\in\mathbb{N}.$$
If $a$ is bounded then convergence in moments implies convergence in distribution.

The following proposition is straightforward and will be used frequently in the paper. A sequence of polynomials $\{P_n=\sum^l_{i=0}c_{n,i} x^i\}_{n>0}$ of degree at most $l\geq k$ is said to converge to a polynomial $P=\sum^k_{i=0}c_i x^i$ of degree $k$ if $c_{i,n}\to c_i$ for $0\leq i\leq k$ and $c_{i,n}\to 0$ for $k<i\leq l$.

\begin{prop}\label{convegence} Suppose that the the sequence of random variables $\{a_n\}_{n>0}$ converges in moments to $a$ and the sequence of polynomials $\{P_n\}_{n>0}$ converges to P. Then the random variables $P_n(a_n)$ converges to $P(a)$.
\end{prop}

In this work we will only consider the $C^*$-probability spaces $(\mathcal{M}_n,\varphi_1)$, where $\mathcal{M}_n$ is the set of matrices of size $n\times n$ and for a matrix $M\in\mathcal{M}_n$ the functional $\varphi_{1}$ evaluated in $M$ is given by $$\varphi_{1}(M)=M_{11}.$$

Let $G = (V,E,1)$ be a finite rooted graph with vertex set $\{1, ..., n\}$ and let $A_G$ be the adjacency matrix. We denote by $A(G)\subset \mathcal{M}_n$
be the adjacency algebra, i.e., the $*$-algebra generated by $A_G$. 

It is easy to see that the $k$-th moment of $A$ with respect to the $\varphi_{1}$  is given the the number of walks in $G$ of size $k$ starting and ending at the vertex $1$. That is, 
$$\varphi_{1}(A^k)=|\{(v_1,...,v_k): v_1=v_k=1 ~ and ~ (v_i,v_{i+1})\in E\}|.$$

Thus one can get combinatorial information of $G$ from the values of $\varphi_1$ in elements of $A(G)$ and vice versa.

Let us  recall the free central limit theorem for free product of graphs (see, e.g. \cite{ALS}) which follows from the usual free central limit theorem for random variables \cite{Voi}.
\begin{thm}[Free Central Limit Theorem for Graphs]\label{fcltg}
 Let $G=(V,E,e)$ be a finite connected graph. Let $A_N$ be the adjacency matrix of the $N$-fold free power $G^{*N}$, and let $\sigma$ be the number of neighbors of $e$ in the graph $G$. Then
 the distribution with respect to the vacuum state of $(N\sigma)^{-1/2}A_N$ converges in moments (and thus weakly) as $N\to\infty$ to the semicircular law.
\end{thm}

For the rest of the paper we define an order which will become handy when estimating vanishing terms in Sections \ref{four} and \ref{five}.

\begin{defi} \label{order} Let $A$ and $B$ be matrices (possibly infinite), we define the order $A\succeq B$ if $A_{ij}\geq B_{ij}$ for all entries $ij$.
\end{defi}

\begin{rem}\label{order2} 1) $\varphi_1(A^k)\geq\varphi_1(B^k)$ if $A\succeq B$.\\
2) For $G_1$ and $G_2$ graphs with $n$ vertices, $G_2$ is a subgraph of $G_1$ iff $A_{G_1}\succeq A_{G_2}$.\\
3) If $A\succeq B$ and  $C\succeq D$ implies $AC\succeq BD$.
\end{rem}

\subsection{Kesten-McKay Distribution}\label{kestenmc}

As we know, by the free central limit theorem, if we have a sequence of $d$-regular trees, then the limiting spectral
distribution of the sequence, as $d\to\infty$, converges to a semicircular law. However, if $d$ is fixed, and we consider
a sequence of $d$-regular graphs, such that the number of vertices tends to infinity, then the limiting spectral distribution
is not semicircular. These limiting spectral distributions, which are known as the Kesten-McKay distributions,
were found by McKay \cite{Mc} while studying properties of $d$-regular graphs and by Kesten \cite{Kes} in his works on random walk on (free) groups.

Let $d\geq 2$ be an integer, we define Kesten-McKay distribution, $\mu_d$, by the density
\begin{equation}\label{kestendistribution} \mathbf{d}\mu_d=\frac{d\sqrt{4(d-1)-x^2}}{2\pi(d^2-x^2)}\mathbf{d}x.\end{equation}
The orthogonal polynomials and the Jacobi parameters of these distributions are well known. More precisely, for $d\geq 2$, the polynomials defined by
 $$T_0(x)=1,\ \ \ T_1(x)=x,$$
and the recurrence formula
\begin{equation}\label{Tk} xT_k(x)=T_{k+1}(x)+(d-1)T_{k-1}(x),\end{equation}
are orthogonal with respect to the distribution $\mu_d$. Thus, it follows that the Jacobi parameters of $\mu_d$ are given by
 $$\beta_m=0,\ \forall m\geq 0\ \ \ \text{and}\ \ \ \gamma_0=d,\ \gamma_n=d-1\ \forall n\geq 1.$$

\begin{rem}
 If we define the following polynomials
$$\tilde{T}_k(x)=\left\{\begin{array}{ll}1,&k=0\\ \sqrt{\frac{d-1}{d}}P_k(x)-\frac{1}{\sqrt{d(d-1)}}P_{k-2}(x),&k=1,2,3,\dots,\end{array}\right.$$
then, $T_k(x)=\tilde{T}_k(x/2\sqrt{d-1}).$
\end{rem}

In Section \ref{random} we will generalize the following theorem due to McKay \cite{Mc} which gives a connection between large $d$-regular graphs and Kesten-McKay distributions.

\begin{thm}
 Let $X_1,\ X_2,\ \dots$ be a sequence of regular graphs with degree $d\geq 2$ such that $n(X_i)\to\infty$ and $c_k(X_i)/n(X_i)\to 0$
 as $i\to\infty$ for each $k\geq 3$, where $n(X_i)$ is the order of $X_i$ and $c_k(X_i)$ is the number of $k$-cycles in $X_i$.
 Then, the limiting distribution for the eigenvalues $X_i$ as $i\to\infty$ is given by $\mu_d$.
\end{thm}

\section{Distance-$k$ graph of $d$-regular trees}

The $d$-regular tree is the $d$-fold free product graph of $K_2$, the complete graph with two vertices. Before we consider asymptotic behavior of the general case of the free product of graphs, we study the distance-$k$ graph of a $d$-regular tree for fixed $d$ and $k$. This is an example where we can find the distribution with respect to the vacuum state in a closed form. Moreover, this example sheds light on the general case of the $d$-fold free product of graphs, in the same way as the $d$-dimensional cube was the leading example for investigations of the distance-$k$ graph of the $d$-fold Cartesian product of graphs (Kurihara \cite{K}).

As a warm up and base case, we calculate the distribution of the distance-$2$ graph with respect to the vacuum state.

For $d\geq2$,  let $A_d^{[k]}$ be the adjacency matrix of distance-$k$ graph of $d$-regular tree. We will sometimes omit the subindex $d$ in the notation and write $A^{[1]}=A$ . Then we have the following equality, which expresses $A^2$ in terms of the distance-$2$ graph and the identity
matrix (see Figure \ref{fig2}). :
\begin{equation}\label{second}A^2= {\color{blue}A_d^{[2]}}+{\color{red}dI}.\end{equation} 
Since $A_d^{[2]}=A^2-dI$ then the distribution is given by the law of  $x^2-d$, where $x$ is a random variable obeying the Kesten-McKay distribution of parameter $d$, $\mu_d$.

\begin{figure}[here]
\begin{center}
\includegraphics[height=5cm]{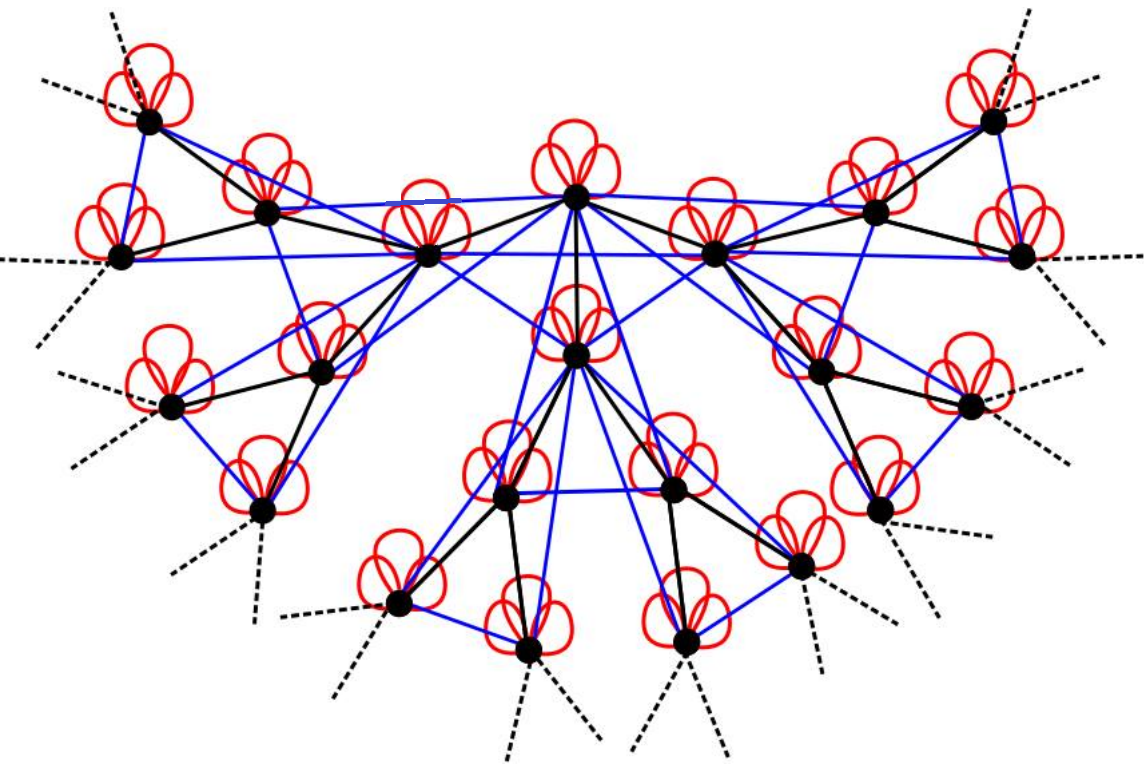}
\caption{\label{fig2} Graph of $A^2$ split in two parts $A^2= {\color{blue}A_d^{[2]}}+{\color{red}dI}$.   }
\end{center}
\end{figure}

For $k\geq 2$ we have the following recurrence formula.
\begin{lem}
Let $d\geq 1$ fixed, then $A^{[1]}=A$, $A^{[2]}=A^2-dI$, and
\begin{equation}\label{recurrence} AA^{[k]}=A^{[k+1]}+(d-1)A^{[k-1]}\ \ \ k=2,\ \dots, d-1.\end{equation}
\end{lem}

\begin{proof}
Let $i$ and $j$ be vertices of the $d$-regular tree, $Y_d$. We have the following three cases.\\
\textbf{Case 1.} If $\partial (i,j)=k+1$ then $(A^{[k]}A)_{ij}=1$, that is because, in this case, there is only one way to get from vertex $j$ to 
vertex $i$. Indeed, since this $Y_d$ is a tree there is only one walk from $i$ to $j$ of size $k+1$ in $Y_d$. Thus, there is exactly one neighbor $l$ of $j$ at distance $k$ from $i$ and thus the only way to go across the distance-$k$ graph and after across $Y_d$ to reach $j$ is trough $l$.\\
\textbf{Case 2.} When we have $\partial(i,j)=k-1$, then $(A^{[k]}A)_{ij}=d-1$. In fact, for the vertex $i$ there are $d-1$ ways
to arrive to $j$ from a neighbor of $j$ at distance $k$ from $i$. Thus, if we are in vertex $i$, there are $d-1$ ways to travel across the distance-$k$ graph and finally go down one level in the $d$-regular tree to vertex $j$, .\\
\textbf{Case 3.} Suppose $|\partial(i,j)-k|\neq 1$, then $(A^{[k]}A)_{ij}=0$. To see this, we just note that, in the $d$-regular
tree we can go up one-level or go down one-level, after going across the distances-$k$ graph, this means that the distance between
$i$ and $j$ would be $k-1$ or $k+1$, which is a contradiction. Therefore if $|\partial(i,j)-k|\neq 1$, there is no way to go from
the vertex $i$ to the vertex $j$, going across the distance-$k$ graph and after, across the $d$-regular tree in one step.\\
Thanks to the above, we obtain the following recurrence formula
\begin{equation} \label{recurrence3}
A^{[k]}A=A^{[k+1]}+(d-1)A^{[k-1]}.
\end{equation}
From the equations (\ref{second}) and (\ref{recurrence3}) we can see that $A^{[k]}$ is a polynomial in $A$ for $k\geq 1$, and 
thus commutes with $A$. Then we can rewrite equation (\ref{recurrence3}) in the more convenient way as follows
\begin{equation*}AA^{[k]}=A^{[k+1]}+(d-1)A^{[k-1]}.\end{equation*}
\end{proof}

Now we can calculate the distribution of the distance-$k$ graph of the $d$-regular tree, for $d$ fixed, which is exactly Theorem \ref{Tuno}.
\begin{proof}[Proof of Theorem 1.1]

From equation (\ref{recurrence}) we can see that $A_d^{[k]}$ fulfills the same recurrence formula than $T_k$ in (\ref{Tk}). Since $A$ is distributed as the Kesten-McKay distribution $\mu_d$, we arrive to the conclusion.
\end{proof}

To end this section we observe  that from the considerations above, by letting $d$ approach infinity, we may find the asymptotic behavior of the distribution of the distance-$k$ graph of the $d$-regular tree. The same behavior is expected when changing the $d$-regular tree with the $d$-fold free product of any finite graph. We will prove this in Section \ref{five} of the paper.
 
\begin{thm}
 For $d\geq 2$, let $A^{[k]}_d$ be the adjacency matrix of the distance-$k$ graph of the $d$-regular tree. Then the distribution with
 respect to the vacuum state of $d^{-k/2}A^{[k]}_d$ converges in moments as $d\to\infty$ to the probability distribution of
 \begin{equation}
  P_k(s),
 \end{equation}
 where $P_k(s)$ is the Chebychev polynomial of degree $k$ and $s$ is a random variable obeying the semicircle law.
\end{thm}

\begin{proof}
If we divide the equation (\ref{recurrence}) by $d^{(k+1)/2}$ we obtain
$$\frac{A_d}{d^{1/2}}\frac{A^{[k]}_d}{d^{k/2}}=\frac{A^{[k+1]}_d}{d^{(k+1)/2}}+\frac{A^{[k-1]}_d}{d^{(k-1)/2}}-\frac{1}{d}\frac{A^{[k-1]}_d}{d^{(k-1)/2}}$$
We write $X=\frac{A_d}{d^{1/2}}$, then we have
$$P^{(1)}(X)=X,\ \ \ P^{(2)}(X)=X^2-I,$$
and the recurrence
$$XP^{(k)}(X)=P^{(k+1)}(X)+P^{(k-1)}(X)-\frac{1}{d}P^{(k-1)}(X),$$
which when $d\to\infty$ becomes the recurrence formula $$XP^{(k)}(X)=P^{(k+1)}(X)+P^{(k-1)}(X).$$
Thus $P^{(k)}(x)$ and $P_k(x)$ satisfy the same recurrence formula asymptotically and 
thanks to the free central limit theorem for graphs (Theorem \ref{fcltg}) we have the convergence, $X\overset{m}\longrightarrow s$. Consequently, combining these two observations and using Lemma \ref{convegence} we obtain the proof. 
\end{proof}

\section{Distance-$2$ graph of free products} \label{four}

In this section we derive the asymptotic spectral distribution of the distance-$2$ graph of the $n$-free power of a graph when $n$ goes to infinity.

In order to analyze the distance-$2$ graphs we give a simple, but useful, decomposition of the square of the adjacency matrix.
\begin{lem}\label{decomp}
Let $G$ be a simple graph with adjacency matrix A, we have the following decomposition of $A^2$:
\begin{equation}\label{distance-2} A^2=\tilde A^{[2]}+D+\Delta,\end{equation}
where $D$ is diagonal with $(D)_{ii}=deg(i)$, $(\Delta)_{ij}=|\text{triangles in $G$ with one side $(i,j)$}|$ and $(\tilde A^{[2]})_{ij}=|\text{paths of size $2$ from $i$ to $j$}|$, whenever $(A^{[2]})_{ij}=1$ and $(\tilde A^{[2]})_{ij}=0$ if $(A^{[2]})_{ij}=0$.
\end{lem}

\begin{proof}Indeed $(A^2)_{ij}$ is zero if the distance between $i$ and $j$ is greater than $2$. So $(A^2)_{ij}>0$
implies that $\partial(i,j)=0,1$ or $2$. If $\partial(i,j)=0$ then $i=j$ and since $(A^2)_{ii}=deg(i)$ we get $D$,
a diagonal matrix with $(D)_{ii}=deg(i)$. Next, if $\partial(i,j)=1$ then each path of size $2$ which forms a triangle with
side $(i,j)$ will contribute to $(A^2)_{ij}=(\Delta)_{ij}$ where $(\Delta)_{ij}=|\text{triangles in $G$ with one side
$(i,j)$}|$. Finally if $\partial(i,j)=2$ then $(A^2)_{ij}$ equals the number of paths of size $2$ from $i$ to $j$, 
which is non-zero exactly when $(\tilde A^{[2]})_{ij}>0$. \end{proof}

\begin{rem}
Notice in Lemma \ref{decomp}, that when $G$ is a tree then $\Delta=0$, $\tilde{A}^{[2]}=A^{[2]}$, therefore $ A^{[2]}=A^2-D.$
\end{rem}

Let $G=(V,E,e)$ be a rooted graph, $A_n=A_{G^{*N}}$ and define $D_n$ and $\Delta_n$ by the decomposition (\ref{decomp})  applied to $G^{*N}=G*\cdots *G$, i.e. 
\begin{equation}\label{distance-2} A_n^2=\tilde A_n^{[2]}+D_n+\Delta_n.\end{equation}
We will describe the asymptotic behavior of each of these matrices. First, we consider the diagonal matrix $D_n$.
\begin{lem}\label{lema2} $D_n/n \to I deg(e)$ entrywise and in distribution w.r.t. the vacuum state. 
\end{lem}
\begin{proof}
For any $i\in G_n$ $(D_n)_{ii}=deg_{G_n}(i)=c_i+(n-1) deg(e)$ for some $0<c_i<max deg(G).$
Thus, $$\frac{(D_n)_{ii}}{n} =\frac{c_i}{n}+\frac{(n-1) deg(e)}{n}\to deg(e).$$
\end{proof}

In order to consider the other matrices in the decomposition we will use the order $\succeq$ from Definition \ref{order}.

\begin{lem} \label{vanishing moments} The mixed moments of   $A_n^2/n$ and $\Delta_n/n$ asymptotically vanish.
\end{lem}

\begin{proof}

Note that the free product does not generate new triangles other than the ones in copies of the original graph. Thus, for $c=\text{max deg}(G)$ the relation $cA_n\succeq\Delta_n$ holds. Hence, for $m_1,\ m_2,\ \dots,\ m_s,\ l_1,\ l_2,\ \dots,\ l_s\in \nn$  and $l_1>0$, from Remark \ref{order2}, we have that
\[\begin{split}&\ \varphi_1\left[\left(\frac{A_n^2}{n}\right)^{m_1}\left(\frac{\Delta_n}{n}\right)^{l_1}\cdots \left(\frac{A^{2}}{n}\right)^{m_s}\left(\frac{\Delta_n}{n}\right)^{l_s}\right]\\
&\leq c^{\sum_i l_i}\varphi_1\left[\left(\frac{A_n^2}{n}\right)^{m_1}\left(\frac{A}{n}\right)^{l_1}\cdots \left(\frac{A^{2}}{n}\right)^{m_s}\left(\frac{A}{n}\right)^{l_s}\right].
\end{split}
\]
From Theorem \ref{fcltg} we have that $A^2/n$ and $A/n^{1/2}$ converge, then the right hand side of the preceding inequality converges to zero as $n$ goes to infinity.
\end{proof}

Since $\tilde A_n^{[2]}$ and $D_n$ are subgraphs of $A_n^2$ we have the following.

\begin{cor}\label{cormix}
The mixed moments of the pairs  $(\tilde A_n^{[2]}/n$, $\Delta/n,)$  and $(D_n/n$, $\Delta/n,)$ asymptotically vanish.
\end{cor}

Finally, we consider the matrix $\tilde A^{[2]}$. 
\begin{lem}\label{lemacon}  $\tilde A_n^{[2]}$ converges to $A_n^{[2]}$ as $n$ goes to infinity.
\end{lem}

\begin{proof}

Observe that we can write $A_n^{[2]}$ as 
$$\tilde A_n^{[2]}=A_n^{[2]}+ \square_n,$$
where for $(i,j)$ at distance $2$ in $G^{*n}$,  the entry $(\square_n)_{ij}$ exceeds in one the number of vertices $k$ such that $i\sim k$ and $k\sim j$. 

We will extend $G$ in the following way. For each $(i,j)$ such that $\square_{ij}$ is positive we put the edge $ij$ and call this new graph $G(ext)$. Now notice that, by construction,  $\Delta_{G(ext)^{*n}}\succeq\square$ and $A_{G(ext)^{*n}}\succeq A_{G^{*n}}$ . Finally, by the previous lemma the mixed  moments of $\Delta_{G(ext)^{*n}}$ and  $A^2_{G(ext)^{*n}}$ asymptotically vanish.  But  $A^2_{G(ext)^{*n}}\succeq A_n^{[2]}$, so the  mixed  moments of  $A_n^{[2]}$ and $\square_n$ also vanish in the limit. This of course means that  $\tilde A_n^{[2]}$ and $A_n^{[2]}$ are asymptotically equal in distribution.

\end{proof}

We have shown that asymptotically $D_n/n$ approximates $I$, $\tilde A_n^{[2]}$ approximates $A_n^{[2]}$ and that the joint moments between
$\tilde A_n^{[2]}$ or $D_n$ and $\Delta_n$ vanish. Thus, we arrive to the following theorem.

\begin{thm}\label{k=2}
The asymptotic distributions of distance-$2$ graph of the $n$-fold free product graph, as $n$ tends to infinity, is given by the distribution of $s^2-1$, where $s$ is a semicircle. \end{thm}

\begin{proof}
 From the equation (\ref{distance-2}), and thanks to Lemmas \ref{convegence}, \ref{lema2}, \ref{lemacon},
 Corolary \ref{cormix} and Theorem \ref{fcltg} we have
 $$A_n^{[2]}\overset{D}\longrightarrow  \tilde A_n^{[2]}\overset{D}\longrightarrow  A_n^2-D_n-\Delta_n\overset{D}\longrightarrow A_n^2-I\overset{D}\longrightarrow s^2-1.$$
\end{proof}

\section{Distance-$k$ graphs of free products}\label{five}

This section contains the proof of Theorem \ref{T2} which describes the asymptotic behavior of the distance-$k$ graph of the $d$-fold free power of graphs. 
We will show that the adjacency matrix satisfies in the limit the recurrence formula (\ref{chev}).
 We start by showing a decomposition similar to the one seen above for $d$-regular trees which plays the role of Lemma \ref{decomp} in the last section.

\begin{thm}\label{teodescomp}
 Let $G$ be a simple finite graph, let $N,k\in \mathbb{N}$ with $N\geq 2$ and  $k\geq 3$ and let $A=A_N$ denote the adjacency matrix of $G^{*N}$.
 Then, we have de following recurrence formula
 \begin{equation}\label{decomp2}A^{[k]}A=\tilde A^{[k+1]}+(N-1)\textnormal{deg}(e)A^{[k-1]}+D^{[k-1]}_N +\Delta^{[k]}_N,\end{equation}
 where $(\tilde A^{[k+1]})_{ij}=|\{l\sim j:\partial(i,l)=k\}|$ whenever $\partial(i,j)=k+1$,\\
 $(D_N^{[k-1]})_{ij}=|\{l\sim j:\partial(i,l)=k, \textnormal{and $j$ and $l$ are in the same copy of $G$}\}|$ if $\partial(i,j)=k-1$ and
 $(\Delta_N^{[k]})_{ij}=|\{l\sim j:\partial(i,l)=k \}|$ when  $\partial(i,j)=k.$
 
\end{thm}

\begin{proof}
 It's easy to see that $(A^{[k]}A)_{ij}$ is zero if $|\partial(i,j)-k|\geq 2$. So
 $(A^{[k]}A)_{ij}>0$ implies that $\partial(i,j)=k-1,\ k$ or $k+1$.

Notice that for each neighbor $l$ of $j$ at distance $k$ from $i$, there is one edge from $i$ to $l$ in $A^{[k]}$ and one from $l$ to $j$ in $A$. Thus each of these neighbors adds $1$ to $(A^{[k]}A)_{ij}$ and there is no further contribution.

First, if $\partial(i,j)=k-1$ there are two types of neighbors $l$ at distance $k$ in $G^{*N}$. The first ones come from the $(N-1)$ copies of $G$ in $G^{*N}$ which have $j$ as a root and contribute to the matrix $A^{[k-1]}$ by $(N-1)\textnormal{deg}(e)$ and the second ones in  which $j$ is in the same copy that $l$, which contribute to $D_N^{[k-1]}$.

Secondly, if $\partial(i,j)=k$ and $(A^{[k]}A)_{ij}>0$ is the number of neighbors of $j$ which are at distance $k$ from
 $i$, then we get $\Delta_N^{[k]}$.
 
Finally, if we have $\partial(i,j)=k+1$, so there exists at least one minimal path from $i$ to $j$, which contains itself a neighbor of $j$ which is at distance $k$ from $i$, therefore
this path contributes to $\tilde A^{[k+1]}$. 

\end{proof}

\subsection*{Proof of Theorem \ref{T2}} We now proceed to prove Theorem \ref{T2} in various steps. While the steps are very similar as the one for the case $k=2$ there are some non trivial modifications to be done for the general case.

We will use induction over $k$. First, observe that for $k=2$, we obtained the conclusion in the last section. Now, suppose
that the fact holds for all $l\leq k$. In order to complete the proof we need the following lemmas and corollaries.
\begin{lem}\label{6.2}
 The mixed moments of $A^{[k]}A/N^{\frac{k+1}{2}}$ and $\Delta_N^{[k]}/N^{\frac{k+1}{2}}$ asymptotically vanish.
\end{lem}

\begin{proof}
  By definition, since the  free product does not generate new cycles, 
 $$\Delta_N^{[k]}\preceq \text{max deg}(G)A^{[k]}.$$
 Hence, for $m_1,\ m_2,\ \dots,\ m_s,\ n_1,\ n_2,\ \dots,\ n_s\in \nn$  and $n_1>0$
 \[
  \begin{split}
   &\varphi_1\left(\left(\frac{A^{[k]}A}{N^{\frac{k+1}{2}}}\right)^{m_1}\left(\frac{\Delta_N^{[k]}}{N^{\frac{k+1}{2}}}\right)^{n_1}\cdots\left(\frac{A^{[k]}A}{N^{\frac{k+1}{2}}}\right)^{m_l}\left(\frac{\Delta_N^{[k]}}{N^{\frac{k+1}{2}}}\right)^{n_l}\right)\\
   &\leq (\text{max deg})^{\sum_i n_i}\varphi_1 \left(\left(\frac{A^{[k]}A}{N^{\frac{k+1}{2}}}\right)^{m_1}\left(\frac{A^{[k]}}{N^{\frac{k+1}{2}}}\right)^{n_1}\cdots\left(\frac{A^{[k]}A}{N^{\frac{k+1}{2}}}\right)^{m_l}\left(\frac{A^{[k]}}{N^{\frac{k+1}{2}}}\right)^{n_l}\right).\\
  \end{split}
 \]
By induction hypothesis $\left(\frac{A^{[k]}A}{N^{\frac{k+1}{2}}}\right)$ and $\left(\frac{A^{[k]}}{N^{\frac{k}{2}}}\right)$ converge  and, therefore,
the right hand side of the last inequality goes to zero.
\end{proof}
Since $\tilde A^{[k+1]}$ and $ D_N^{[k-1]}$ are subgraphs of $A^{[k]}A$, the following is a direct consequence of the previous lemma.
\begin{cor}\label{6.3}
 The mixed moments of $\left(\tilde A^{[k+1]}/N^{\frac{k+1}{2}},\ \Delta_N^{[k]}/N^{\frac{k+1}{2}}\right)$ and \\$\left(\Delta_N^{[k]}/N^{\frac{k+1}{2}},D_N^{[k-1]}/N^{\frac{k+1}{2}}\right)$ asymptotically
 vanish.
\end{cor}

\begin{cor}\label{6.4}
 The matrices $\Delta_N^{[k]}/N^{\frac{k+1}{2}}$ and $D_N^{[k-1]}/N^{\frac{k+1}{2}}$ converge to zero as $N$ tends to infinity.
\end{cor}
\begin{proof}
 In the proof of Lemma \ref{6.2},  by taking $m_i=0$ for all $0\leq i\leq s$ we obtain the conclusion for $\Delta_N^{[k]}/N^{\frac{k+1}{2}}$.
Using $A^{[k-1]}$ instead $A^{[k]}$ the same proof works for $D^{[k-1]}/N^{\frac{k+1}{2}}$.
\end{proof}
 In the proof of the next lemma, we shall use the following extension of a graph $G$. For $k\geq 2$ we define $G_{ext(k)}$ as the 
 graph which contains the graph $G$, and if $G$ has a cycle of even length smaller than $2k$, we add all the possible edges between the vertices of this 
 cycle. Note that $G_{ext(2)}=G_{ext}$. It is important to notice the fact that 
 $$(G_{ext(k)})^{*N}=(G^{*N})_{ext(k)}.$$
\begin{lem}\label{6.5}
 Let $k\geq 2$, then $\lim_{N\to\infty}\frac{\tilde{A}^{[k+1]}}{N^{\frac{k+1}{2}}}-\frac{A^{[k+1]}}{N^{\frac{k+1}{2}}}=0$.
\end{lem}

\begin{proof}
 Let $i,j\in \underset {s\in [N]} * V$ be such that $\left(\tilde A^{[k+1]}\right)_{ij}>0$. Let 
 $$C^{k+1}_{ij}=\{\text{cycles of even length in a path of length $k+1$ from $i$ to $j$}\},$$
 notice that 
 $$|C^{k+1}_{ij}|\leq (\text{max deg}(G))^{k+1}.$$
Here, is important to observe that the right side bound does not depend on $i$, $j$ neither $N$, because the free product of graph
does not produce new cycles. Then we can write
 \begin{equation}\label{ultimolema}\tilde A^{[k+1]}-A^{[k+1]}\preceq (\text{max deg}(G))^{k+1}\left(A^{[k]}_{G_{ext(k+1)}}+A^{[k-1]}_{G_{ext(k+1)}}+\cdots+A_{G_{ext(k+1)}}\right).\end{equation}
 Then, we obtain from (\ref{ultimolema})
 \[
 \begin{split}
  &\left(\frac{\tilde A^{[k+1]}-A^{[k+1]}}{N^{(k+1)/2}}\right)\\&\preceq (\text{max deg}(G))^{k+1}\left(\frac{A^{[k]}_{G_{ext(k+1)}}}{N^{(k+1)/2}}+\frac{A^{[k-1]}_{G_{ext(k+1)}}}{N^{(k+1)/2}}+\cdots+\frac{A_{G_{ext(k+1)}}}{N^{(k+1)/2}}\right)\\
  &=(\text{max deg}(G))^{k+1}\left(\frac{A^{[k]}_{G_{ext(k+1)}}}{N^\frac{k}{2}}\frac{1}{N^{1/2}}+\frac{A^{[k-1]}_{G_{ext(k+1)}}}{N^\frac{k-1}{2}}\frac{1}{N}+\cdots+\frac{A_{G_{ext(k+1)}}}{N^\frac{1}{2}}\frac{1}{N^\frac{k}{2}}\right).\\
 \end{split}
\]
By induction hypothesis we have that $\left(A^{[i]}_{G_{ext(k)}}/N^{i/2}\right)$ converges for all $i\leq k$. Therefore we have
$$\left(\frac{\tilde A^{[k+1]}-A^{[k+1]}}{N^{(k+1)/2}}\right)\underset {N\to\infty}\longrightarrow 0,$$
which completes the proof.
\end{proof}

Now we can finish the proof of Theorem \ref{T2}. From (\ref{decomp2}) we have that
\begin{equation}\label{s}
\frac{A_N^{[k+1]}}{(\text{deg}(e)N)^{\frac{k+1}{2}}}=\frac{A_N^{[k]}A_N}{(\text{deg}(e)N)^{\frac{k+1}{2}}}
-\frac{A_N^{[k-1]}}{(\text{deg}(e)N)^{\frac{k-1}{2}}}-C(N,k+1),
\end{equation}
where 
\begin{equation*}
 C(N,k+1)=\frac{\text{deg}(e)A_N^{[k-1]}+\Delta_N^{[k]}+D_N^{[k-1]}-\left(\tilde A^{[k+1]}-A^{[k-1]}\right)}{(\text{deg}(e)N)^{\frac{k+1}{2}}}.
\end{equation*}
Due to the induction hypothesis we have, $\text{deg}(e)A_N^{[k-1]}/(\text{deg}(e)N)^{\frac{k+1}{2}}$ converging to zero, furthermore
by Corollary \ref{6.4} and Lemma \ref{6.5} $$\frac{\Delta_N^{[k]}+D_N^{[k-1]}-\left(\tilde A^{[k+1]}-A^{[k-1]}\right)}{(\text{deg}(e)N)^{\frac{k+1}{2}}},$$
converges to zero. Hence
\begin{equation}\label{sss} C(N,k+1)\longrightarrow 0.\end{equation}
Finally, using the induction hypothesis we can see that
\begin{equation}\label{ss} \frac{A_N^{[k]}A_N}{(\text{deg}(e)N)^{\frac{k+1}{2}}}-\frac{A_N^{[k-1]}}{(\text{deg}(e)N)^{\frac{k-1}{2}}}\longrightarrow P_k(s)s-P_{k-1}(s)=P_{k+1}(s),\end{equation} where the last equality is given by (\ref{chev}).
Thus, mixing (\ref{s}) with (\ref{sss}) and (\ref{ss}) we obtain that $$\frac{A_N^{[k+1]}}{(\text{deg}(e)N)^{\frac{k+1}{2}}}\longrightarrow P_{k+1}(s).$$

\section{$d$-Regular Random Graphs}\label{random}

Apart from the Erdos-Renyi models \cite{E1,E2}, possibly, the random $d$-regular graphs are possibly the most studied and well understood random graphs. 

In the original paper by McKay \cite{Mc}, he proved that the asymptotical spectral distributions of $d$-regular random graph are exactly the ones appearing in (\ref{kestendistribution}). Heuristically, the reason is that, locally,  large random $d$-regular graphs look like the $d$-regular tree and thus asymptotically their spectrum should coincide. This turns out to remain true for the distance-$k$ graph and thus we shall expect to get a similar result.  In this section we formalize this intuition.

Let $X$ be a $d$-regular graph with vertex set $\{1,2,\dots,n(X)\}$. For each $i\geq 3$ let $c_i(X)$ be the number of cycles of length $i$. Let $A^{[k]}(X)$ be the adjacency matrix of the distance-$k$ graph of $X$. The following is a generalization of the main theorem in McKay \cite{Mc}. 

\begin{thm}\label{d-regular graphs}
 For $d\geq 2$ fixed, let $X_1,\ X_2,\ \dots$ be a sequence of $d$-regular graphs such that $n(X_i)\to\infty$ and  $c_j(X_i)/n(X_i)\to 0$ as $i\to\infty$ for each $j\geq 3$. Then the distribution with respect to the normalized trace of $A^{[k]}(X_i)$ converges in moments, as $i\to \infty$, to the  distribution of $A^{[k]}_d$  with respect to the vacuum state.
\end{thm}

\begin{proof}
We follow the original proof of McKay\cite{Mc}  with simple modifications. Let $n_r(X_i)$ denote the number of vertices $v$ of $X_i$ such that the subgraph of $X_i$ induced by the vertices at distance at most $r=mk$, where $m\in\nn$, from $v$ is acyclic. By hypothesis we have that $n_r(X_i)/n(X_i)\to 1$ as $i\to\infty$. 
The number of closed walks of length $m$ in the distance-$k$ graph of the $d$-regular graph starting at each such vertex is $\varphi(A_d^{[k]})$. For each of the remaining vertices the number of closed walks of length $m$ is certainly less than $d^r$.
Then, for each $m$, there are numbers $\hat\varphi_m(X_i)$ such that $0\leq\hat\varphi_m(X_i)\leq d^r,$ and
\[ 
 \begin{split}
 \varphi_{tr}\left(A^{[k]}(X_i)\right)&=\frac{\varphi(A^{[k]}_d)n_r(X_i)}{n(X_i)}+\frac{(n(X_i)-n_r(X_i))\hat\varphi_m(X_i)}{n(X_i)}\\
 &\longrightarrow \varphi(A^{[k]}_d)\ \ \ \ \ \text{as}\ i\to\infty.
 \end{split}
\]
 \end{proof}

Fix $d>0$. Let $s_1 <s_2, < . . $  be the sequence of possible cardinalities of regular graphs with degree $d$. For each $n$ define $R_n$ to be the set of all labeled regular graphs with degree n and order $s_i$.

In order to consider the $d$-regular uniform random graphs we use the following lemma of Wormald \cite{ Wo2}.
\begin{lem}\label{Worm}
For each $k>3$ define $c_{k,n}$ to be the  average number of $k$-cycles in the members of $R_n$. Then for each $k$, $c_{k,n}\to(d-1)^k/2k$ as $n\to \infty$. 
\end{lem}

\begin{proof}[Proof of Theorem \ref{T3}]
Consider a graph $G_n$ which consist a disjoint union of the all the labelled graphs of size $s_n$. The eigenvalue distribution of $G_n$ coincides with the expected eigenvalue distribution of $R_n.$  Now by Lemma \ref{Worm}, $G$ satisfies the assumptions of Theorem \ref{d-regular graphs} and thus we arrive to the theorem.
\end{proof}

\end{document}